\documentclass{amsart}
\usepackage{amssymb}
\usepackage{a4wide}
\usepackage[all]{xypic}

\DeclareMathOperator{\Aut}{Aut}
\DeclareMathOperator{\Gl}{GL}
\DeclareMathOperator{\Id}{Id}
\DeclareMathOperator{\Hom}{Hom}

\newcommand{\N}{{\mathbb N}}
\newcommand{\Z}{{\mathbb Z}}

\newcommand{\ds}{\displaystyle}

\newtheorem{theorem}{Theorem}[section]
\newtheorem{proposition}[theorem]{Proposition}
\newtheorem{lemma}[theorem]{Lemma}

\newtheorem{corollary}[theorem]{Corollary}

\begin{document}

\title{The $R_\infty$ property for abelian groups}
\author[K.\ Dekimpe]{Karel Dekimpe}
\author[D.\ Gon\c{c}alves]{Daciberg Gon\c{c}alves}
\address{KU Leuven Kulak\\
E. Sabbelaan 53\\
8500 Kortrijk\\
Belgium}
\email{karel.dekimpe@kuleuven-kulak.be}
\address{Departamento de Matem\'atica - IME, Universidade de S\~ao Paulo\\ Caixa Postal 66.281 - CEP 05314-970\\ 
S\~{a}o Paulo\\
Brasil}
\email{dlgoncal@ime.usp.br}

\begin{abstract}

It is well known there is no finitely generated abelian group which has the $R_\infty$ property. We will show that also many non-finitely generated abelian groups 
do not have the $R_\infty$ property, but this does not hold for all of them! In fact we construct an uncountable number of infinite countable abelian groups  
which do have the $R_{\infty}$ property. 
We also construct an abelian group such that the cardinality of the Reidemeister classes is uncountable for any automorphism of that group.  
\end{abstract}

\maketitle
\section{Introduction}

Let $G$ be a group and $\varphi$ be an endomorphism of $G$. Then two elements 
$x,y$ of $G$ are said to be Reidemeister equivalent (with respect to $\varphi)$, if there exists an element 
$z\in G$ such that $y = z x \varphi(z)^{-1}$. The equivalence classes  are called the Reidemeister classes or twisted conjugacy classes. 


\medskip

{\it Definition:  The Reidemeister number of a homomorphism $\varphi$, denoted by  
$R(\varphi)$, is the cardinality  of the Reidemeister classes of $\varphi$}. 
We remark here that most authors define the Reidemeister number as either a positive integer or $\infty$. This latter definition of course coincides with ours in the finite case, but does not allow to make a distinction between the 
various infinite cases.

\medskip

The Reidemeister number is a relevant ingredient in connection with many parts of mathematics. See for example \cite {FLT} and references therein. This is for instance also the case in the  study  of the fixed point properties of the homotopy class of a self map on a topological space. In this situation, the group $G$ will be the fundamental group $\pi_1(X)$ of the space and the 
homomorphism $\varphi=f_\sharp$ is the one which is induced by the map $f$ on the fundamental group $G$. Under certain hypothesis the 
Reidemeister number $R(\varphi)$ is then exactly the number of essential  fixed point 
classes of $f$ if $R(\varphi)$ is finite and the number of essential  fixed point 
classes of $f$ is zero if  $R(\varphi)$ is infinite. See \cite{j} and \cite{w} and the references 
therein for more information.

\medskip

A group $G$ has the   
$R_\infty$ property if for every automorphism $\varphi$ of $G$ the Reidemeister number is not 
finite. In recent years many works have studied the question of 
which groups $G$ have the 
$R_\infty$ property.  We refer to \cite{FLT} for an overview of the results which have been obtained in this direction. 
The present work will also give a contribution for this problem,  where 
we will consider infinite abelian groups. If an abelian group $A$ is finitely generated then 
it is well known that $A$ does not have the   $R_\infty$ property, since it is easy to see that the automorphism 
$\varphi:A \rightarrow A:a \mapsto -a$ has a finite Reidemeister number in this case. So, in this paper, we will focus on abelian groups which are not 
finitely generated.   For information about infinite abelian groups in general  we refer to 
\cite{fu1}, \cite{fu2} and \cite{ka}.

\medskip

To the best of our knowledge, up till now, there is no example in literature of  an abelian group having the $R_{\infty}$ property. In this paper we do construct  an uncountable number of countable abelian groups  which do have the $R_{\infty}$ property.

\medskip

Before we announce the main results of this paper, let us fix some notation
\begin{itemize}
\item Let $p$ be a prime, then with $\Z_p$, we will denote the additive group of $p$-adic integers.
\item For any positive integer $n\geq 2$, $\Z/n\Z$ will denote the additive group of integers modulo $n$.
\item Let ${\mathcal P}$ be any set of primes, then $\Z_{\mathcal P}$ denotes the additive group of rational numbers which can be written as a fraction whose denominator is relative prime with all primes in ${\mathcal P}$. When 
$p$ is a prime, then $\hat{p}$ is the set of all primes which are different from $p$ and hence $\Z_{\hat{p}}$ is the group 
of all rational numbers whose denominator is a power of $p$. 
\item Finally, when $p$ is a prime $\Z(p^\infty)$ is the Pr\"ufer group $\ds\frac{\Z_{\hat{p}}}{\Z}$.  
\end{itemize}

\medskip

Recall that a group $G$ is divisible if and only if for any $x\in G$ and any positive integer $n$, there is a $y\in G$ such
that $y^n=x$.

\medskip

We can now formulate the  main results of this note:

\medskip

{\bf Proposition} \ref{notR} The following abelian groups do not 
have the $R_{ \infty}$ property:
\begin{enumerate}
\item Abelian divisible groups. 
\item The groups $\Z_{\mathcal P}$ for any set of primes $\mathcal{P}$ 
\item The $p$-adic  groups $\Z_p$ for any prime $p$.
\item  Any abelian torsion group without $2$-torsion elements.
\end{enumerate}

\medskip

In the formulation of the following result we use the 
Reidemeister spectrum of a group $A$ which is the set $\{ R(\varphi)\;|\; \varphi\in \Aut(A)\}$:

\medskip

{\bf Proposition} \ref{spect}  For any prime $p\ne 2$, the  spectrum of $\Z_{\hat{p}}$ is 
\[\{2\} \cup \{ p^m+1 \;|\; m \in \N\} \cup \{ p^m-1 \;| \;m \in \N\} \cup\{  \infty\}.\]
Hence if     $\varphi: \Z_{\hat{p}} \to\Z_{\hat{p}}$ is an automorphism and $p\ne 2$, then $R(\varphi)\ne 1$.

\medskip

{\bf Theorem} \ref{countex} Let  $\mathcal P$ be an  infinite set of primes and consider the group  
$\displaystyle\bigoplus \Z_{\hat p}$
as $p$ runs over the set $\mathcal P$.  Then any automorphism of this group has infinite Reidemeister number.

\medskip

{\bf Theorem} \ref{uncountex} Let  $\mathcal P$ be an  infinite set of primes and consider the group  
$\displaystyle \prod  \Z_{\hat p}$
as $p$ runs over the set $\mathcal P$.  Then any automorphism of this group has the property that the set of Reidemeister classes is uncountable. In particular this group also has the $R_{\infty}$ property.

\bigskip

This work is divided into 3 sections besides the introduction. In section 2 we recall a few elementary properties of the infinite abelian groups.  In section 3 we show that many infinite abelian groups do not have the $R_{\infty}$ property. Groups constructed using standard constructions like direct sums and direct products are analysed. In section 4 we provide examples of countable and uncountable abelian groups
which have the $R_{\infty}$ property. Finally we present an example of a group having the property that 
for any automorphism the  Reidemeister number is always uncountable.

\section{Preliminaries about infinite abelian groups}
In this section we recall some known results  about infinite abelian groups and prove some elementary facts about these groups which are used in our study.  Let $A$ be an abelian group and $\varphi:A \to A$ a 
homomorphism of $A$. Whenever we need to have 
$\varphi$  an automorphism we make this explicit.

From \cite{ka}, Theorem 3, page 9 we have:

 \begin{theorem}\label{div} Any  abelian group $A$ has a unique largest divisible subgroup $M$, and $A=M \oplus N$ where $N$ has no non-zero divisible subgroups.
\end{theorem}

This theorem shows the relevance of the divisible groups for the description of the infinite abelian groups. 
An abelian group having no non-zero divisible subgroup is called a reduced group. When $M$ is the maximal 
divisible subgroup of an abelian group $A$ we will call $A/M$ the reduced part of $A$. 

\medskip

 Because 
the groups in question are abelian groups, it follows that the Reidemeister number of an endomorphism $\varphi$ of 
such an abelian group $A$ coincides with the cardinality of the quotient group $A/{\rm Im}(\varphi-\Id_A)$
(or $A/{\rm Im}(\Id_A-\varphi)$).  

\medskip

Now we prove a lemma which is on the one hand very simple but on the other hand 
very useful to  show that  many abelian groups do not have the $R_{\infty}$ property. 

\begin{lemma}\label{ind} Let $A$ be an abelian group and consider the homomorphism $2: A \to A: a\mapsto 2(a) = a+ a$ (so multiplication by $2$).
 Then: 
 \begin{enumerate}
\item The Reidemeister number of this  homomorphism is 1.
\item If $2(A)$, the image of the homomorphism, has finite index in $A$, then the automorphism $\tau: A \to A$ given by $x\mapsto -x$ has Reideimeister number equal to the index of $2(A)$ in $A$.
\end{enumerate}

\end{lemma}

\begin{proof} Part (1) follows straightforward from the definition of the Reidemeister classes since 
for any $a$ we have  $a= -a+ 0 - 2(-a)=a$ so $a$ is in the same Reidemeister class as
   $0$ for any $a$. The second part follows from the fact that for abelian groups  the Reidemeister classes for $\tau$  correspond  with the cosets of the image of the homomorphism $\Id-\tau:A\to A: a \mapsto a-(-a)=2a$.
\end{proof}

\begin{corollary}\label{corR1} If multiplication by $2$ is an automorphism, then not only does $A$ not have the $R_{\infty}$ property, but also  $A$ admits  automorphisms, which are multiplication by $-1$ and multiplication by $2$,  which have Reidemeister number $1$.
\end{corollary}

\begin{proof} Follows promptly from the lemma above.
\end{proof}

\medskip

In the rest of this paper we will also need the following lemma. 
\begin{lemma}\label{indexm} Let ${\mathcal P}$ be any set of primes and let $m>1$ be a positive integer whose prime decomposition only consists of primes in ${\mathcal P}$. 
Then, the index  $[\Z_{\mathcal P}:m \Z_{\mathcal P}]$ equals 
$m$. 
\end{lemma}

\begin{proof} It suffices to  show that any element of $\Z_{\mathcal P}$ belongs to exactly one of the cosets
\[ i + m \Z_{\mathcal P} \mbox{ with } i\in \{ 0,1,2,\ldots, m-1\}.\]
Let $x\in \Z_{\mathcal P}$. If $x=0$, then $x\in 0+ m \Z_{\mathcal P}$, otherwise 
$x=\ds\frac{q}{r}$ where $r$ is 1 or a product of primes not belonging to ${\mathcal P}$ and $q\in \Z$.
As $\gcd(m,r)=1$, it follows from  B\'ezout's identity that there exists integers $\alpha$ and $\beta$ with 
$q = \alpha r + \beta m$. Then
\[ x+ m  \Z_{\mathcal P}  = \frac{q}{r} + \Z_{\mathcal P} = \alpha + m \frac{\beta}{r} + m \Z_{\mathcal P}  = 
\alpha +m \Z_{\mathcal P} .\]
Now, write $\alpha = i + m \alpha'$ for some $i\in \{ 0,1,2,\ldots, m-1\}$ and $\alpha'\in \Z$. 
It follows immediately that 
\[ x\in x+ m \Z_{\mathcal P} = \alpha+ m \Z_{\mathcal P} = i  + m \Z_{\mathcal P}.\]
so that $x$ belongs to at least one of these cosets $i+m \Z_{\mathcal P}$. It is also easy to see that all of these cosets are different, which finishes the proof.
\end{proof}

\section{abelian groups which do not have the $R_{\infty}$ property}
 In this section we show that many abelian groups do not have the $R_{\infty}$ property and 
in some cases in fact we compute  the Reidemeister spectrum  (i.e.\ the set of all possible cardinals 
which are the Reidemeister number for  some automorphism of the group).  The calculation of the spectrum is useful for section 4.

\subsection{Divisible groups, the $p$-adic integers and torsion groups}

\begin{proposition}\label{notR} The following abelian groups do not 
have the $R_{ \infty}$ property:
\begin{enumerate}
\item  abelian divisible groups. 
\item The groups $\Z_{\mathcal P}$, where  
${\mathcal P}$ is any 
subset of the set of all primes. 
\item  The $p$-adic integers $\Z_p$ for any prime  $p$.
\item Any abelian torsion group without $2$-torsion elements.
\end{enumerate}
 \end{proposition}

\begin{proof} Part (1) follows promptly from Lemma \ref{ind} item (2). \\
 Part (2) follows from Corollary~\ref{corR1} if
$2\not \in  {\mathcal P}$ and from Lemma \ref{ind} item (2) and Lemma \ref{indexm} otherwise. \\
 Part (3) follows from Corollary~\ref{corR1} if
$p$ is odd  and from Lemma \ref{ind} item (2) for $p=2$. \\
Part (4) follows from Corollary~\ref{corR1}.
\end{proof}

Note that in case $\mathcal{P}$ is the set of all primes, then  the group $\Z_{\mathcal P}$ is exactly the 
group $\Z$, which certainly does not have the $R_{\infty}$ property, but this group is finitely generated.

\medskip

For divisible groups we can even say more:

\begin{proposition}
Let $A$ be a divisible abelian group and $\varphi:A\rightarrow A$ be any homomorphism. If $R(\varphi)$ is finite,
then $R(\varphi)=1$.
\end{proposition}
\begin{proof}
If $R(\varphi)$ is finite, then the group ${\rm Im}(\varphi - \Id_A)$ is a subgroup of finite index in $A$. However, the only 
subgroup of finite index in a divisible group $A$ is the group $A$ itself. Therefore $A= {\rm Im}(\varphi - \Id_A)$ and
hence $R(\varphi)=1$.
\end{proof}

In fact, divisible groups can be totally ignored when studying the $R_\infty$ property of abelian groups. We make this precise in the following proposition.

\begin{proposition}\label{reduced part}
Let $A$ be an abelian group. Then $A$ has the $R_{\infty}$ property if and only if the reduced part of $A$ has the $R_{\infty}$ property. 
\end{proposition} 

\begin{proof}
Let $M$ be the unique maximal divisible subgroup of $A$, then $A=M\oplus N$ where $A/M\cong N$ is the reduced part of $A$. Let $\varphi$ be any automorphism of $A$, then $\varphi$ restricts to an automorphism $\varphi'$ of 
$M$ and induces an automorphism $\bar{\varphi}$ of the quotient $A/M$. 

\medskip

It is easy to see that when $R(\bar\varphi)$ is infinite, then also $R(\varphi)$ is infinite. Hence, if $A/M$ 
has the $R_{\infty}$ property, then also $A$ has the $R_\infty$ property.

\medskip

On the other hand, assume that $A$ has the $R_\infty$ property and consider any automorphism $\bar\varphi$ of 
$A/M$. We can lift this automorphism, to an automorphism $\varphi$ of $A$ by defining 
\[ \varphi: M\oplus N \rightarrow M\oplus N: (m,n) \mapsto (-m, \bar\varphi(n)).\]
Recall that  $R(\varphi)$ equals the index of Im$(\Id_A - \varphi)$ in $A$. Since 
\[ {\rm Im}(\Id_A - \varphi)= 2M \oplus {\rm Im}(\Id_N - \bar\varphi) =  M \oplus {\rm Im}(\Id_N - \bar\varphi)\]
we have that $R(\varphi)=R(\bar\varphi)$ and so this Reidemeister number is infinite, since $A$ has the 
$R_\infty$ property. Hence $A/M$ also has this property.
\end{proof}

It follows that, from the point of view of the $R_\infty$ property, we are left to the study 
of reduced abelian groups.

\medskip

Also in the case the groups are torsion, it suffices to study the $2-$torsion groups. Indeed,
any abelian torsion group $A$ can be decomposed as a direct sum $\ds A=\bigoplus_{p\;{\rm prime}} A_p$, where 
$A_p$ is the $p$-primary part of $A$, i.e.\ the subgroup of $A$ all elements of $p$--power order (\cite[Theorem~1]{ka}).
As all of these subgroups $A_p$ are characteristic in $A$, we have that $\Aut(A)=\ds \prod_{p\;{\rm prime}} \Aut(A_p)$. 
Since for any $p\neq2$ there is an automorphism $\varphi_p\in \Aut(A_p)$ with Reidemeister number $R(\varphi_p)=1$ (e.g.\ $\varphi_p$ is multiplication with 2), 
it follows that $A$ has property $R_\infty$ if and only if $A_2$ has property $R_\infty$.

\medskip

{\bf Remark:} We do not know an example of an abelian $2-$torsion group which has the $R_{\infty}$ property. 

\medskip

In section~\ref{torsion} we continue our study of torsion groups.

\subsection{Direct sum and product of any abelian group}

\begin{proposition}\label{many factors} If $A$ is an arbitrary abelian group, then for any finite integer $n>1$,
there is an automorphism $\varphi: A^n \to A^n$, 
that has Reidemeister number 1.\\
Furthermore, in the case $\alpha$ is an infinite cardinal,  the same result holds  for both  the direct sum $\ds \bigoplus_\alpha A$ (weak direct product) and  the direct product $\ds \prod_\alpha A$.\\ 
Hence none of these groups has the $R_\infty$ property.
 \end{proposition}

\begin{proof}  If $n$ is either $2$ or $3$, it is easy to find an  element $\theta_n\in \Gl(n,\Z)$  such that 
$det(\theta_n-\Id_n)=1$, i.e.\ it has Reidemeister number $1$. E.g.\ we can take 
\[ \theta_2 = \begin{pmatrix} 1 & 1 \\ -1 & 0\end{pmatrix} \mbox{ and }\theta_3=
\begin{pmatrix}
0 & -1 & 0 \\ 1 & 0 & 1 \\ 0 & 1 & 1
\end{pmatrix}. \]
Then given an arbitrary 
integer $n>1$, using the result for $n=2$ and $n=3$,  we can construct a blocked diagonal element 
$\theta_n\in \Gl(n,\Z)$  such that 
$det(\theta_n-\Id_n)=1$, i.e.\ it has Reidemeister number $1$.  Now we use this matrix, in the obvious way,
  to define an automorphism $\varphi$ of $A^n$. Then the homomorphisms $\varphi-\Id$ is surjective and 
then we have $R(\varphi)=1$. 

\medskip

Now let $\alpha$ be an infinite cardinal. Then $\alpha=\alpha+\alpha$ and hence 
$\ds\prod_{\alpha} A\cong\prod_{\alpha}(A\oplus A)$.
On any factor $A\oplus A$, we can then consider the automorphism $\psi$ which is given by the matrix 
$\theta_2$. Using this $\psi$, we define an automorphism $\varphi$ of $\ds \prod_\alpha(A\oplus A)$ 
which is given by
\[ \varphi=\prod_\alpha \psi: \prod_\alpha(A\oplus A) \rightarrow \prod_\alpha(A\oplus A): (a_j,b_j)_{j\in \alpha} \mapsto 
\psi (a_j,b_j)_{j\in \alpha} .\]
Again the homomorphism $\varphi-\Id$ is surjective 
and the result follows for the direct product. \\
The case of the direct sum is completely analogous.
 \end{proof}

\subsection{The subgroups of the rationals}
Now we compute for any prime $p$ the spectrum of the group $\Z_{\hat{p}}$,  the integers localized at the  set of primes $\hat{p}$. As already mentioned before, this is 
the set of fractions where the denominators are powers of $p$. 
We already saw that these groups do not have the $R_\infty$ property (Proposition \ref{notR}),  but the calculation of the spectrum will be useful for section 4.

\begin{proposition}\label{spect}  
The spectrum of $\Z_{\hat{p}}$ is 
\[ \{2\} \cup \{ p^m+1 \;|\; m \in \N\} \cup \{ p^m-1 \;| \;m \in \N\} \cup\{  \infty\},\]
in case $p\neq 2$ and is 
\[  \{ 2^m+1 \;|\; m \in \N\} \cup \{ 2^m-1 \;| \;m \in \N\} \cup\{  \infty\},\]
for $p=2$.\\
Hence if   $\varphi: \Z_{\hat{p}} \to \Z_{\hat{p}}$ is an automorphism and $p\ne 2$, then $R(\varphi)\ne 1$.
\end{proposition}

\begin{proof}
Let us consider an automorphism $\varphi$ of the group  $\Z_{\hat{p}}$. Then $\ds\varphi(1) = \frac{a}{p^n}$ for some integers $a$ and $n$ and $\varphi$ is  just multiplication with $\ds\frac{a}{p^n}$. 
Since $\varphi$ is an automorphism, $1$ must be in the image of $\varphi$, and hence there must exist a 
$\ds \frac{b}{p^k}\in \Z_{\hat{p}}$ such that $\ds \varphi(\frac{b}{p^k}) = 
\frac{b}{p^k}\frac{a}{p^n}= 1$. Hence, the only prime which possibly divides $a$ is the prime $p$. 
It follows that $\varphi(1) = \pm p^m$ for some integer $m$. Since $R(\varphi) = R(\varphi^{-1})$, we may assume that 
$m\geq 0$. Recall that $R(\varphi)$ is the index of  Im($\varphi - \Id$) in $\Z_{\hat{p}}$. We distinguish four cases:
\begin{itemize}
\item $\varphi(1)=1 $ (first case where $m=0$). In this case $\varphi-\Id$ is the zero homomorphism and $R(\varphi)=\infty$.
\item $\varphi(1)=-1 $ (second case where $m=0$). In this case $\varphi-\Id$ is multiplication by $-2$. 
For $p=2$, this is an automorphism of $\Z_{\hat{2}}$, which leads to $R(\varphi)=1=2^1-1$. When $p\neq 2$, it follows from 
Lemma~\ref{indexm} that $R(\varphi)=2$.
\item $\varphi(1)=p^m$ (with $m>0$). Now  $\varphi-\Id$ is multiplication with $p^m-1$ and since $p^m-1$ is  relative 
prime to $p$, Lemma~\ref{indexm} implies that $R(\varphi)= p^m-1$.
\item $\varphi(1)=-p^m$ (with $m>0$). Then Im$(\varphi-\Id)={\rm Im}(\Id - \varphi)$ and since $\Id - \varphi$ is the 
same as multiplication with $p^m+1$, Lemma~\ref{indexm} again implies that $R(\varphi)= p^m+1$.
\end{itemize} 
This finishes our computation of the spectrum.\\
The fact that in case $p\ne 2$, we always have that $R(\varphi)\ne 1$ is clear from  the first part.
\end{proof}

\subsection{Torsion groups and direct products}\label{torsion}

Among torsion groups we have the Pr\"ufer groups $\Z({p^{\infty}})$ where $p$ is any prime number.
From section 5, {\it Divisible groups}, in  \cite{ka} it follows that the Pr\"ufer groups are divisible so they do 
not have  the $R_{\infty}$ property by Proposition \ref{notR}.

\medskip

Another easy way of constructing torsion groups is to take direct sums of finite cyclic groups. 
This situation is completely dealt with in the next two propositions.

 
\begin{proposition}
Let $n_1,n_2,n_3,n_4,\ldots $ be an increasing set of positive integers. Then 
both 
\[ \bigoplus_{i\in \N}  \Z/2^{n_i} \Z \mbox{ \ \ and \ \ } \prod _{i\in \N}  \Z/2^{n_i} \Z\]
admit an automorphism with Reidemeister number equal to 1.\\
In particular these groups and also the torsion subgroup of $\ds  \prod _{i\in \N}  \Z/2^{n_i} \Z$ do not have 
the $R_\infty$ property.
\end{proposition}
\begin{proof}
We begin this proof with the case of the direct product.\\
A general element of $A=\ds \prod _{i\in \N}  \Z/2^{n_i} \Z$ can be written in the form 
\[ (a_1 + 2^{n_1} \Z, a_2 + 2^{n_2} \Z, a_3 + 2^{n_3} \Z, a_4+2^{n_4} \Z.\ldots).\]
for some integers $a_1,a_2,a_3,a_4,\ldots$. For simplicity we will write this shortly as 
\[ ( \overline{a_1}, \overline{a_2}, \overline{a_3}, \overline{a_4},\ldots )\]
Now define $\varphi:A\rightarrow A$ by 
\[ ( \overline{a_1}, \overline{a_2}, \overline{a_3}, \overline{a_4},\ldots )\mapsto
(\overline{a_1+a_2+a_3}, \overline{a_2+a_3}, \overline{a_3+a_4+a_5}, \overline{a_4+a_5},
\overline{a_5+a_6+a_7}, \overline{a_6+a_7},\ldots )\]
So the $(2 k-1)$-th component of this image is $\overline{ a_{2k-1} + a_{2k} + a_{2k+1}}$ and the 
$2k$-th component is $\overline{a_{2k}+a_{2k+1}}$.\\
As by assumption $n_1\leq n_2\leq n_3 \leq n_4 \leq \cdots$, the map $\varphi$ is well defined and is 
an endomorphism of $A$. 
In fact, $\varphi$ is an automorphism, since it is easy to check that the map $\psi:A\rightarrow A$:
\[ ( \overline{a_1}, \overline{a_2}, \overline{a_3}, \overline{a_4},\ldots )\mapsto
(\overline{a_1-a_2}, \overline{a_2-a_3+a_4}, \overline{a_3-a_4}, \overline{a_4-a_5+a_6},
\overline{a_5-a_6}, \overline{a_6-a_7+a_8},\ldots )\]
is also well defined and is an endomorphism which is the inverse of $\varphi$. Moreover the map $\varphi- \Id:A\rightarrow A$ is given by 
\[ ( \overline{a_1}, \overline{a_2}, \overline{a_3}, \overline{a_4},\ldots )\mapsto
(\overline{a_2+a_3}, \overline{a_3}, \overline{a_4+a_5}, \overline{a_5},
\overline{a_6+a_7}, \overline{a_7},\ldots )\]
which is clearly surjectice. Hence $R(\varphi)=1$.

\medskip

It is clear that one can use the restriction of $\varphi$ to the direct sum or the torsion subgroup of  
$A$ to obtain the same result in these cases.
\end{proof}
 
\begin{proposition} Let $A$ be any direct sum or any direct product of finite cyclic groups.
Then $A$ does not  have the  $R_{\infty}$ property.
\end{proposition}

\begin{proof} 
We will give the proof of the fact that any direct product of finite cyclic groups does not have
the $R_\infty$  property. The case for the direct sum is completely analogous and is left to the reader.

As any finite cyclic group is the direct product of cyclic $p$--groups (for different primes $p$), we can assume 
that  
\[ A = \prod_{i\in I} A_i\]
for some index set $I$ and each $A_i$ is a cyclic group of prime-power  order. 
Now we divide $I$ into two disjoint subsets $I=I_1\cup I_2$, where 
\[ I_1=\{ i \in I \;|\; A_i \mbox{ is a 2--group }\}\mbox{ and }I_2= I\backslash I_1.\]
Let $A^{(1)}= \ds \prod_{i\in I_1} A_i$ and $A^{(2)}=\ds \prod_{i\in I_2} A_i$. Then 
$A= A^{(1)} \times A^{(2)}$. Note that multiplication by 2 is an automorphism, say $\varphi_2$, of  $A^{(2)}$
with Reidemeister number $R(\varphi_2)=1$. It is now enough to show that also $A^{(1)}$ admits an 
automorphism $\varphi_1$ with finite Reidemeister number, for then the automorpshim $\varphi=\varphi_1 \times 
\varphi_2$ will have Reidmeister number $R(\varphi)=R(\varphi_1)\times R(\varphi_2) = R(\varphi_1)$.  

So from now onwards we concentrate on $A^{(1)}$, and 
for any positive integer $n\in \N$ we let 
\[ I_1^{(n)} = \{ i \in I_1\; |\; \# A_i= 2^n \} \mbox{ \ \ and \ \ } A_n^{(1)} = \prod_{i \in I_1^{(n)}} A_i.\]
and so $A^{(1)}=\ds \prod_{n\in \N} A_n^{(1)} $. 
For those $n\in \N$ for which $\# I_1^{(n)}>1$, we know, by  Proposition~\ref{many factors}, that there exists an automorphism $\varphi_1^{(n)}$ of 
$A_n^{(1)} $ with Reidemeister number $R(\varphi_1^{(n)})=1$. 

Now, let $N\subseteq \N$ be the subset of positive integers $n$ such that $\# I_1^{(n)} =1$ and 
let $i_n$ denote the unique element in $I_1^{(n)}$. Then 
$\ds A^{(1)} = \prod_{n\in N} A_{i_n} \times \prod_{n \in\N\backslash N} A_n^{(1)}$.
As $\ds \prod_{n \in\N\backslash N} \varphi_1^{(n)}$ is an automorphism of $\ds  \prod_{n \in\N\backslash N} A_n^{(1)}$ with Reidmeister number 1, it suffices to find an automorphism of 
$\ds \prod_{n\in N} A_{i_n}$ of finite Reidemeister number. If $N$ is a finite set then any automorphism (e.g. the identity) will do. When $N$ is infinite, the result follows from the previous proposition.
\end{proof}

\section{abelian groups which have the $R_{\infty}$ property }
In this section we present for the first time an example of an abelian group which has the $R_{\infty}$ property. In fact we will show that there are at least an uncountable number of abelian groups with this property.

Recall that for any prime $p$, $\Z_{\hat{p}}$ is the subgroup of the rationals consisting of all fractions whose denominator is a power of $p$.

\begin{lemma} \label{hom zero}If $p_1\ne p_2$ (both primes) then $\Hom(\Z_{\hat p_1}, \Z_{\hat p_2})$ contains only one element,  which is the trivial homomorphism. 
 \end{lemma}

\begin{proof} Given a homomorphism $\varphi\in \Hom(\Z_{\hat p_1}, \Z_{\hat p_2})$ 
this homomorphism is determined completely by the value of $\varphi(1)$. 
But $1\in \Z_{\hat p_1}$ is divisible by $p_1^n$ for all $n$, hence also 
$\varphi(1)$ must be divisible by all powers $p_1^n$. As there is no non-zero 
element in   $\Z_{\hat p_2}$ with this property ($p_1\ne p_2$), we must have that $\varphi(1)=0$. So the result follows.
 \end{proof}

\begin{theorem}\label{countex} Let  ${\mathcal P}$ be an  infinite set of primes and consider the group 
\[  A=\bigoplus_{p\in {\mathcal P}}  \Z_{\hat p}.\]
Then $A$ has the $R_\infty$ property.
\end{theorem}

\begin{proof}  It follows from Lemma~\ref{hom zero} that  $\ds \Aut(A)=\prod_{p\in {\mathcal P}} \Aut(\Z_{\hat{p}})$ i.e.\  
any automorphism $\varphi$ of $A$ can be decomposed 
as a direct product $\varphi= \ds\prod_{p \in {\mathcal P}}\varphi_p$ of (auto)morphisms $\varphi_p: \Z_{\hat p} \to  \Z_{\hat p}$.
From the previous section (Proposition~\ref{spect}) any automorphism of  $\Z_{\hat p}$ has Reidemeister number greater than 1 if $p\ne 2$. Since an infinite number of primes is different from $2$, it follows that $R(\varphi)=\infty$.

This shows that $A$ has the $R_\infty$ property.
\end{proof}

\begin{corollary} There is an uncountable number of abelian groups which have the $R_\infty$ property.
 \end{corollary}

\begin{proof} Given two distinct sets of primes ${\mathcal P}$ and 
${\mathcal P}'$, the corresponding groups $\ds \bigoplus_{p\in {\mathcal P}}  \Z_{\hat p}$  and
$\ds \bigoplus_{p\in {\mathcal P'}}  \Z_{\hat p}$
are not isomorphic.  This follows using similar arguments as in  Lemma~\ref{hom zero} above. As there are an uncountable number of infinite subsets of the set of all primes, the result follows from 
Theorem~\ref{countex} above. 
\end{proof}

Observe that for a given infinite set ${\mathcal P}$ of primes we can also construct the group  $\ds \prod_{p\in {\mathcal P}}  \Z_{\hat p}$ 
instead of $\ds \bigoplus_{p\in {\mathcal P}} Z_{\hat p}$. A similar result hold for this group where now we can even say that 
the cardinality of the set of Reidemeister classes is indeed uncountable.

\begin{theorem}\label{uncountex} Let  ${\mathcal P}$ be an  infinite set of primes and consider the group  $\ds \prod_{p\in {\mathcal P}}  \Z_{\hat p}$.\\
Then for any automorphism $\varphi$ of this group, the set of Reidemeister classes is uncountable. In particular this group also has the $R_{\infty}$ property.
\end{theorem}

{\bf Remark: } It is easy to extend this results to obtain abelian groups of any infinite cardinality with the $R_\infty$ property. Indeed, when taking a direct
sum $M\oplus \ds \bigoplus_{p\in {\mathcal P}}  \Z_{\hat p}$ or $M\oplus  \prod_{p\in {\mathcal P}}$ where $M$ is a divisible group (of any cardinality you like), the 
resulting group again has the $R_\infty$ property by  Proposition~\ref{reduced part}.

\end{document}